\documentclass{article}
\usepackage{array,amsmath,amsthm}
\numberwithin{equation}{section}
\usepackage{amssymb}
\usepackage{indentfirst}
\usepackage{geometry}
\begin{document}

\newtheorem{thm}{Theorem}[section]
\newtheorem{cor}[thm]{Corollary}
\newtheorem{prop}[thm]{Proposition}
\newtheorem{conj}[thm]{Conjecture}
\newtheorem{lem}[thm]{Lemma}
\newtheorem{Def}[thm]{Definition}
\newtheorem{rem}[thm]{Remark}
\newtheorem{prob}[thm]{Problem}
\newtheorem{ex}{Example}[section]

\newcommand{\be}{\begin{equation}}
\newcommand{\ee}{\end{equation}}
\newcommand{\ben}{\begin{enumerate}}
\newcommand{\een}{\end{enumerate}}
\newcommand{\beq}{\begin{eqnarray}}
\newcommand{\eeq}{\end{eqnarray}}
\newcommand{\beqn}{\begin{eqnarray*}}
\newcommand{\eeqn}{\end{eqnarray*}}
\newcommand{\bei}{\begin{itemize}}
\newcommand{\eei}{\end{itemize}}

\makeatletter
\newcommand{\rmnum}[1]{\romannumeral #1}
\newcommand{\Rmnum}[1]{\expandafter\@slowromancap\romannumeral #1@}
\makeatother

\newcommand{\pa}{{\partial}}
\newcommand{\V}{{\rm V}}
\newcommand{\R}{{\bf R}}
\newcommand{\K}{{\rm K}}
\newcommand{\e}{{\epsilon}}
\newcommand{\tomega}{\tilde{\omega}}
\newcommand{\tOmega}{\tilde{Omega}}
\newcommand{\tR}{\tilde{R}}
\newcommand{\tB}{\tilde{B}}
\newcommand{\tGamma}{\tilde{\Gamma}}
\newcommand{\fa}{f_{\alpha}}
\newcommand{\fb}{f_{\beta}}
\newcommand{\faa}{f_{\alpha\alpha}}
\newcommand{\faaa}{f_{\alpha\alpha\alpha}}
\newcommand{\fab}{f_{\alpha\beta}}
\newcommand{\fabb}{f_{\alpha\beta\beta}}
\newcommand{\fbb}{f_{\beta\beta}}
\newcommand{\fbbb}{f_{\beta\beta\beta}}
\newcommand{\faab}{f_{\alpha\alpha\beta}}

\newcommand{\pxi}{ {\pa \over \pa x^i}}
\newcommand{\pxj}{ {\pa \over \pa x^j}}
\newcommand{\pxk}{ {\pa \over \pa x^k}}
\newcommand{\pyi}{ {\pa \over \pa y^i}}
\newcommand{\pyj}{ {\pa \over \pa y^j}}
\newcommand{\pyk}{ {\pa \over \pa y^k}}
\newcommand{\dxi}{{\delta \over \delta x^i}}
\newcommand{\dxj}{{\delta \over \delta x^j}}
\newcommand{\dxk}{{\delta \over \delta x^k}}

\newcommand{\px}{{\pa \over \pa x}}
\newcommand{\py}{{\pa \over \pa y}}
\newcommand{\pt}{{\pa \over \pa t}}
\newcommand{\ps}{{\pa \over \pa s}}
\newcommand{\pvi}{{\pa \over \pa v^i}}
\newcommand{\ty}{\tilde{y}}
\newcommand{\bGamma}{\bar{\Gamma}}

\title { $(p, q)$-Sobolev inequality and Nash inequality on forward complete Finsler metric measure manifolds \footnote{The first author is supported by the National Natural Science Foundation of China (12371051, 12141101, 11871126).}}
\author{ Xinyue Cheng  and Qihui Ni }
\date{}

\maketitle

\begin{abstract}
 In this paper, we carry out in-depth research centering around the $(p, q)$-Sobolev inequality and Nash inequality on forward complete Finsler metric measure manifolds under the condition that ${\rm Ric}_{\infty} \geq -K$ for some $K \geq 0$.  We first obtain a global $p$-Poincar\'{e} inequality on such Finsler manifolds. Based on this, we can derive a $(p, q)$-Sobolev inequality. Furthermore, we establish a global optimal $(p, q)$-Sobolev inequality with a sharp Sobolev constant. Finally, as an application of the $p$-Poincar\'{e} inequality, we prove a Nash inequality.
\end{abstract}

\maketitle

\section{Introduction}

In $n$-dimensional Euclidean space $\mathbb{R}^{n}$, there is an interesting question as follows: how can one control the size of a function in terms of the size of its gradient? The well-known Sobolev inequalities answer precisely this question. On the real line, the answer is given by a simple yet extremely useful calculus inequality. Namely, for any smooth function $f$ with compact support on the line,
$$
|f(t)| \leq \frac{1}{2} \int_{-\infty}^{+\infty}\left|f^{\prime}(s)\right| ds.
$$
The factor $1/2$ in this inequality comes from the fact that $f$ vanishes at both $+\infty$ and $-\infty$. It is natural to wonder if there is such an inequality for smooth compactly supported functions in higher-dimensional Euclidean spaces. Fortunately, we can prove the following result.  Fix an integer $n \geq 2$ and a real $p, \ 1 \leq p<n$ and set $q= np /(n-p)$. Then there exists a constant $C=C(n, p)$ such that, for any $f \in \mathcal{C}_{0}^{\infty}\left(\mathbb{R}^n\right)$,
\[
\left(\int_{\mathbb{R}^n}|f(x)|^q dx\right)^{1/q}\leq C \left(\int_{\mathbb{R}^n}|\nabla f(x)|^{p} dx\right)^{1/p},
\]
where $\nabla f$ is the gradient of $f$ and $\mathcal{C}_{0}^{\infty}\left(\mathbb{R}^n\right)$ denotes the set of all smooth compactly supported functions in $\mathbb{R}^n$ (\cite{SC}).

In Riemannian setting, the study of Sobolev spaces on Riemannian manifolds is a very important field which has been undergoing great development. Particularly, Sobolev inequalities are very useful when developing analysis on Riemannian manifolds, even more so than on Euclidean space, because other tools such as Fourier analysis are not available any more. This is particularly true when one studies large scale behavior of solutions of partial differential equations such as the Laplace and heat equations. In \cite{SC}, Saloff-Coste shows that Poincar\'{e} inequality and the doubling property of the measure imply a family of local Sobolev inequalities. As applications of Sobolev inequalities, he studies Gaussian heat kernel estimates and the Rozenblum-Lieb-Cwikel inequality and derives elliptic Harnack inequality and parabolic Harnack inequalities. On the other hand, Hebey discusses best Sobolev constants problems for compact Riemannian manifolds and complete noncompact Riemannian respectively in \cite{Heby}.

In \cite{nash}, Nash introduced the Nash inequality on Riemannian manifolds and proved that a Nash inequality implies that the heat diffusion semigroup $(H_{t})_{t>0}$ is ultracontractive. Later,  Nash type inequality has been widely studied and applied (e.g. see \cite{BM,SC}).

Finsler geometry is just Riemannian geometry without the quadratic restriction (\cite{Chern}). It is natural to study and develop Sobolev inequalities and the relevant applications on Finsler metric measure manifolds. In \cite{Ohta1}, Ohta establishes a logarithmic Sobolev inequality and a sharp Sobolev inequality on compact Finsler metric measure manifolds with ${\rm Ric}_N \geq K$ for some $K>0$ and $N \in (-\infty , -2)\cup [n, \infty)$ (also see \cite{Ohta2}). Further, Xia proves the existence of two types of optimal $(p, q)$-Sobolev inequalities on compact Finsler manifolds with ${\rm Ric}_{N} \geq K$ for $N \in[n, \infty)$ and $K \in \mathbb{R}$. In particular, when $K>0$, Xia establishes the sharp $(p,2)$-Sobolev inequality for $2 \leq p \leq 2N /(N-2)$ (see \cite{Xia}).  Recently, the first author and Feng establish two local uniform Sobolev inequalities on forward  complete Finsler metric measure spaces with weighted Ricci curvature ${\rm Ric}_{\infty}$ bounded below (\cite{ChengF1,ChengF2}). On the other hand, Ohta shows  a Nash inequality as an important foundation to derive a non-sharp Sobolev inequality under the condition that ${\rm Ric}_{N} \geq K$ for some $K>0$ and $N \in[n, \infty)$ in \cite{Ohta1} (also see \cite{Ohta2}).

In this paper, we mainly study the existence of global optimal $(p, q)$-Sobolev inequality and Nash inequality on forward complete Finsler metric measure manifolds under the condition that ${\rm Ric}_{\infty} \geq -K$ for some $K \geq 0$. For convenience to introduce our main result, we first give some necessary notations.  We always use $(M, F, m)$ to denote a Finsler manifold $(M, F)$ equipped with a smooth measure $m$ which we call a Finsler metric measure manifold (or Finsler measure space briefly). A Finsler measure space is not a metric space in usual sense because Finsler metric $F$ may be nonreversible, that is, $F(x, y)\neq F(x, -y)$ may happen. This non-reversibility causes the asymmetry of the associated distance function. In order to overcome the deficiencies that a Finsler metric $F$ may be nonreversible, Rademacher defined the reversibility $\Lambda$ of $F$ by
\be
\Lambda:=\sup _{(x, y) \in TM \backslash\{0\}} \frac{F(x,y)}{F(x, -y)}.
\ee
Obviously, $\Lambda \in [1, \infty]$ and $\Lambda=1$ if and only if $F$ is reversible (\cite{Ra}). For $x_1, x_2 \in M$, the distance from $x_1$ to $x_2$ is defined by
\be
d_{F}(x_{1}, x_{2}):=\inf _{\gamma} \int_{0}^{1} F(\gamma (t), \dot{\gamma}(t)) d t,
\ee
where the infimum is taken over all $C^1$ curves $\gamma:[0,1] \rightarrow M$ such that $\gamma(0)=$ $x_1$ and $\gamma(1)=x_2$. Note that $d_{F} \left(x_1, x_2\right) \neq d_{F} \left(x_2, x_1\right)$ unless $F$ is reversible. The diameter of $M$ is defined by
\be
{\rm Diam}(M):=\sup _{x_{1}, x_{2} \in M} \{d_{F}(x_{1}, x_{2})\}.
\ee
Further, let ${\bf S}={\bf S}(x, y)$ be the $S$-curvature of $F$ and
\be
\vartheta := \sup\limits_{(x,y)\in TM\setminus \{0\}}\frac{|{\bf S}(x,y)|}{F(x,y)}. \label{supS}
\ee
For more details, see Section \ref{Introd}.

Our first main result is the following theorem.

\begin{thm}\label{sob5}
Let $(M, F, m)$ be an $n$-dimensional $(n \geq 2)$ forward complete Finsler manifold with finite reversibility $\Lambda$. Assume that ${\rm Ric}_{\infty} \geq -K$ for some $K \geq 0$, $d:= {\rm Diam}(M)< \infty$ and $m_{0}:= m(M)> 0$. Then, for any $\nu > n+1$ and $q\in [1, 2]$, there exists a positive constant $A =A( \Lambda, n, K, m_{0}, d, \vartheta)$ such that for any $u \in W^{1,q}(M)$, the following optimal inequality holds.
\begin{equation}
\left(\int_{M} | u |^{p} \, dm\right)^{\frac{q}{p}} \leq m_{0}^{-\frac{q}{\nu}} \int_{M} | u |^{q} \, dm + A \int_{M} [F^{*}(du)]^{q} \, dm , \label{sob6}
\end{equation}
where $\frac{1}{p} = \frac{1}{q}-\frac{1}{\nu}$.
\end{thm}

An interesting problem is as follows: how to get the best value of Sobolev constant $A$ in (\ref{sob6})? This is an open problem under our assumptions.

Our second main result is the following Nash inequality.

\begin{thm}{\label{nashs}}
Let $(M, F, m)$ be an $n$-dimensional forward complete Finsler manifold with finite reversibility $\Lambda$. Assume that ${\rm Ric}_{\infty} \geq -K$ for some $K \geq 0$, $d := {\rm Diam}(M)< \infty$ and $m_{0}=m(M)>0$. Let  $p \geq 1$. Then, for any $u\in W^{1,p}(M)$, there exist a positive constant $D=D(n, K, \vartheta, \Lambda, d)$ such that
\begin{equation}
\|u\|_{L^{p}}^{2+\frac{2}{n+1}} \leq D(n, K, \vartheta, \Lambda, d)m_{0}^{\frac{2(1-p)}{(n+1)p}}\left(\|F(\nabla u)\|_{L^{p}}^{2}+\|u\|_{L^{p}}^{2}\right)\|u\|_{L^{1}}^{\frac{2}{n+1}}. \label{nash1}
\end{equation}
\end{thm}

The paper is organized as follows. In Section \ref{Introd}, we give some necessary definitions and notations. Then we prove a global $p$-Poincar\'{e} inequality on forward complete Finsler metric measure manifolds under the condition that ${\rm Ric}_{\infty} \geq -K$ for some $K \geq 0$
in Section \ref{smean}. Next, starting from Lemma 5.3 in \cite{Xia} (see Lemma \ref{sobll}), we will give the proof of Theorem \ref{sob5} in Section \ref{Ssob}. Finally, we will prove Nash inequality (\ref{nash1}) by the global $p$-Poincar\'{e} inequality and the volume comparison in Section \ref{nash}.

\section{Preliminaries}\label{Introd}
In this section, we briefly review some necessary definitions, notations and  fundamental results in Finsler geometry. For more details, we refer to \cite{BaoChern, ChernShen, Ohta2, Shen1}.

Let $M$ be an $n$-dimensional smooth manifold. A Finsler metric on manifold $M$ is a function $F: T M \longrightarrow[0, \infty)$  satisfying the following properties: (1) $F$ is $C^{\infty}$ on $TM\backslash\{0\}$; (2) $F(x,\lambda y)=\lambda F(x,y)$ for any $(x,y)\in TM$ and all $\lambda >0$; (3)  $F$ is strongly convex, that is, the matrix $\left(g_{ij}(x,y)\right)=\left(\frac{1}{2}(F^{2})_{y^{i}y^{j}}\right)$ is positive definite for any nonzero $y\in T_{x}M$. The pair $(M,F)$ is called a Finsler manifold and $g:=g_{ij}(x,y)dx^{i}\otimes dx^{j}$ is called the fundamental tensor of $F$. A non-negative function on $T^{*}M$ with analogous properties is called a Finsler co-metric. For any Finsler metric $F$, its dual metric
\be
F^{*}(x, \xi):=\sup\limits_{y\in T_{x}M\setminus \{0\}} \frac{\xi (y)}{F(x,y)}, \ \ \forall \xi \in T^{*}_{x}M. \label{co-Finsler}
\ee
is a Finsler co-metric.

We define the reverse metric $\overleftarrow{F} $ of a Finsler metric $F$ by $\overleftarrow{F}(x, y):=F(x,-y)$ for all $(x, y) \in T M$. It is easy to see that $\overleftarrow{F}$ is also a Finsler metric on $M$. A Finsler metric $F$ on $M$ is said to be reversible if $\overleftarrow{F}(x, y)=F(x, y)$ for all $(x, y) \in T M$. Otherwise, we say $F$ is irreversible.

For a non-vanishing vector field $V$ on $M$, one introduces the weighted Riemannian metric $g_V$ on $M$ given by
\be
g_{V}(y, w)=g_{ij}(x, V_{x})y^{i} w^{j}  \label{weiRiem}
\ee
for $y,\, w\in T_{x}M$. In particular, $g_{V}(V,V)=F^{2}(x,V)$.

Let $x_{0}\in M$. The forward and backward geodesic balls of radius $R$ with center at $x_{0}$ are respectively defined by
$$
B_{R}^{+}(x_0):=\{x \in M \mid d_{F}(x_{0}, x)<R\},\qquad B_{R}^{-}(x_0):=\{x \in M \mid d_{F}(x, x_{0})<R\}.
$$
In the following, we will always denote $B_{R}:=B^{+}_{R}(x_0)$ for some $x_{0}\in M$ for simplicity.

Let $(M,F)$ be a Finsler manifold of dimension $n$. Let $\pi: TM \setminus \{0\} \rightarrow M$ be the projective map.  The pull-back $\pi ^{*}TM$ admits a unique linear connection, which is called the Chern connection. The Chern connection $D$ is determined by the following equations
\beq
&& D^{V}_{X}Y-D^{V}_{Y}X=[X,Y], \label{chern1}\\
&& Zg_{V}(X,Y)=g_{V}(D^{V}_{Z}X,Y)+g_{V}(X,D^{V}_{Z}Y)+ 2C_{V}(D^{V}_{Z}V,X,Y) \label{chern2}
\eeq
for $V\in TM\setminus \{0\}$  and $X, Y, Z \in TM$, where
$$
C_{V}(X,Y,Z):=C_{ijk}(x,V)X^{i}Y^{j}Z^{k}=\frac{1}{4}\frac{\pa ^{3}F^{2}(x,V)}{\pa V^{i}\pa V^{j}\pa V^{k}}X^{i}Y^{j}Z^{k}
$$
is the Cartan tensor of $F$ and $D^{V}_{X}Y$ is the covariant derivative with respect to the reference vector $V$.

Given a non-vanishing vector field $V$ on $M$,  the Riemannian curvature  $R^V$ is defined by
$$
R^V(X, Y) Z=D_X^V D_Y^V Z-D_Y^V D_X^V Z-D_{[X, Y]}^V Z
$$
for any vector fields $X$, $Y$, $Z$ on $M$. For two linearly independent vectors $V, W \in T_x M \backslash\{0\}$, the flag curvature is defined by
$$
\mathcal{K}^V(V, W)=\frac{g_V\left(R^V(V, W) W, V\right)}{g_V(V, V) g_V(W, W)-g_V(V, W)^2}.
$$
Then the Ricci curvature is defined as
$$
\operatorname{Ric}(V):=F(x, V)^{2} \sum_{i=1}^{n-1} \mathcal{K}^V\left(V, e_i\right),
$$
where $e_1, \ldots, e_{n-1}, \frac{V}{F(V)}$ form an orthonormal basis of $T_x M$ with respect to $g_V$.

A $C^{\infty}$-curve $\gamma:[0,1] \rightarrow M$ is called a geodesic  if $F(\gamma, \dot{\gamma})$ is constant and it is locally minimizing. The exponential map $\exp _{x}: T_{x} M \rightarrow M$ is defined by $\exp _x(v)=\gamma(1)$ for $v \in T_x M$ if there is a geodesic $\gamma:[0,1] \rightarrow M$ with $\gamma(0)=x$ and $\dot{\gamma}(0)=v$. A Finsler manifold $(M, F)$ is said to be forward complete (resp. backward complete) if each geodesic defined on $[0, \ell)$ (resp. $(-\ell, 0])$ can be extended to a geodesic defined on $[0, \infty)$ (resp. $(-\infty, 0])$. We say $(M, F)$ is complete if it is both forward complete and backward complete. By Hopf-Rinow theorem on forward complete Finsler manifolds, any two points in $M$ can be connected by a minimal forward geodesic and the forward closed balls $\overline{B_R^{+}(p)}$ are compact (see \cite{BaoChern, Shen1}).

Let $(M, F, m)$ be an $n$-dimensional Finsler manifold with a smooth measure $m$. Write the volume form $dm$ of  $m$ as $d m = \sigma(x) dx^{1} dx^{2} \cdots d x^{n}$. Define
\be\label{Dis}
\tau (x, y):=\ln \frac{\sqrt{{\rm det}\left(g_{i j}(x, y)\right)}}{\sigma(x)}.
\ee
We call $\tau$ the distortion of $F$. It is natural to study the rate of change of the distortion along geodesics. For a vector $y \in T_{x} M \backslash\{0\}$, let $\sigma=\sigma(t)$ be the geodesic with $\sigma(0)=x$ and $\dot{\sigma}(0)=y.$  Set
\be
{\bf S}(x, y):= \frac{d}{d t}\left[\tau(\sigma(t), \dot{\sigma}(t))\right]|_{t=0}.
\ee
$\mathbf{S}$ is called the S-curvature of $F$ (\cite{ChernShen, shen}).

Let $Y$ be a $C^{\infty}$ geodesic field on an open subset $U \subset M$ and $\hat{g}=g_{Y}.$  Let
\be
d m:=e^{- \psi} {\rm Vol}_{\hat{g}}, \ \ \ {\rm Vol}_{\hat{g}}= \sqrt{{det}\left(g_{i j}\left(x, Y_{x}\right)\right)}dx^{1} \cdots dx^{n}. \label{voldecom}
\ee
It is easy to see that $\psi$ is given by
$$
\psi (x)= \ln \frac{\sqrt{\operatorname{det}\left(g_{i j}\left(x, Y_{x}\right)\right)}}{\sigma(x)}=\tau\left(x, Y_{x}\right),
$$
which is just the distortion of $F$ along $Y_{x}$ at $x\in M$ (\cite{ChernShen, Shen1}). Let $y := Y_{x}\in T_{x}M$ (that is, $Y$ is a geodesic extension of $y\in T_{x}M$). Then, by the definitions of the S-curvature, we have
\beqn
&&  {\bf S}(x, y)= Y[\tau(x, Y)]|_{x} = d \psi (y),  \\
&&  \dot{\bf S}(x, y)= Y[{\bf S}(x, Y)]|_{x} =y[Y(\psi)],
\eeqn
where $\dot{\bf S}(x, y):={\bf S}_{|m}(x, y)y^{m}$ and ``$|$" denotes the horizontal covariant derivative with respect to the Chern connection  (\cite{shen, Shen1}). Further, the weighted Ricci curvatures are defined as follows (\cite{ChSh,Ohta2})
\beq
{\rm Ric}_{N}(y)&=& {\rm Ric}(y)+ \dot{\bf S}(x, y) -\frac{{\bf S}(x, y)^{2}}{N-n},   \label{weRicci3}\\
{\rm Ric}_{\infty}(y)&=& {\rm Ric}(y)+ \dot{\bf S}(x, y). \label{weRicciinf}
\eeq
We say that Ric$_N\geq K$ for some $K\in \mathbb{R}$ if Ric$_N(v)\geq KF^2(v)$ for all $v\in TM$, where $N\in \mathbb{R}\setminus \{n\}$ or $N= \infty$.

\vskip 3mm

According to Lemma 3.1.1 in \cite{Shen1}, for any vector $y\in T_{x}M\setminus \{0\}$, $x\in M$, the covector $\xi =g_{y}(y, \cdot)\in T^{*}_{x}M$ satisfies
\be
F(x,y)=F^{*}(x, \xi)=\frac{\xi (y)}{F(x,y)}. \label{shenF311}
\ee
Conversely, for any covector $\xi \in T_{x}^{*}M\setminus \{0\}$, there exists a unique vector $y\in T_{x}M\setminus \{0\}$ such that $\xi =g_{y}(y, \cdot)\in T^{*}_{x}M$ . Naturally,  we define a map ${\cal L}: TM \rightarrow T^{*}M$ by
$$
{\cal L}(y):=\left\{
\begin{array}{ll}
g_{y}(y, \cdot), & y\neq 0, \\
0, & y=0.
\end{array} \right.
$$
It follows from (\ref{shenF311}) that
\be
F(x,y)=F^{*}(x, {\cal L}(y)). \label{preserN}
\ee
Thus ${\cal L}$ is a norm-preserving transformation. We call ${\cal L}$ the Legendre transformation on Finsler manifold $(M, F)$.
Let
$$
g^{*kl}(x,\xi):=\frac{1}{2}\left[F^{*2}\right]_{\xi _{k}\xi_{l}}(x,\xi).
$$
For any $\xi ={\cal L}(y)$, we have
\be
g^{*kl}(x,\xi)=g^{kl}(x,y), \label{Fdual}
\ee
where $\left(g^{kl}(x,y)\right)= \left(g_{kl}(x,y)\right)^{-1}$.

Given a smooth function $u$ on $M$, the differential $d u_x$ at any point $x \in M$,
$$
d u_x=\frac{\partial u}{\partial x^i}(x) d x^i
$$
is a linear function on $T_x M$. We define the gradient vector $\nabla u(x)$ of $u$ at $x \in M$ by $\nabla u(x):=\mathcal{L}^{-1}(d u(x)) \in T_x M$. In a local coordinate system, we can express $\nabla u$ as
\be \label{nabna}
\nabla u(x)= \begin{cases}g^{* i j}(x, d u) \frac{\partial u}{\partial x^i} \frac{\partial}{\partial x^j}, & x \in M_u, \\ 0, & x \in M \backslash M_u,\end{cases}
\ee
where $M_{u}:=\{x \in M \mid d u(x) \neq 0\}$ \cite{Shen1}. In general, $\nabla u$ is only continuous on $M$, but smooth on $M_{u}$. By (\ref{preserN}), $F(x, \nabla u)=F^{*}(x, du)$.

\vskip 3mm

Let $W^{1, p}(M)(p>1)$ be the space of functions $u \in L^{p}(M)$ with $\int_{M}[F(\nabla u)]^{p} d m + \int_{M}\big[\overleftarrow{F}(\overleftarrow{\nabla} u)\big]^{p} d m<\infty$ and $W_{0}^{1, p}(M)$ be the closure of $\mathcal{C}_0^{\infty}(M)$ under the (absolutely homogeneous) norm
\be
\|u\|_{W^{1, p}(M)}:=\|u\|_{L^p(M)}+\frac{1}{2}\|F(\nabla u)\|_{L^p(M)}+\frac{1}{2}\|\overleftarrow{F}(\overleftarrow{\nabla} u)\|_{L^p(M)},
\ee
where $\mathcal{C}_{0}^{\infty}(M)$ denotes the set of all smooth compactly supported functions on $M$ and $\overleftarrow{\nabla} u$ is the gradient of $u$ with respect to the reverse metric $\overleftarrow{F}$. In fact, $\overleftarrow{F}(\overleftarrow{\nabla} u)=F(\nabla(-u))$.

\section{A global $p$-Poincar\'{e} inequality}\label{smean}

In this section, we will prove a global $p$-Poincar\'{e} inequality starting from the following volume comparison result.

\begin{lem}{\rm(\cite{ChengF2})}\label{volume}
Let $(M, F, m)$ be an $n$-dimensional forward complete Finsler metric measure space. Assume that ${\rm Ric}_{\infty} \geq -K$ for some $K \geq 0$. Then, along any minizing geodesic starting from the center $x_{0}$ of $B^{+}_{R}(x_{0})$, we have the following for any $0 < r_{1}<r_{2} <R$
\begin{equation}
\frac{m(B_{r_{2}}(x_{0}))}{m(B_{r_{1}}(x_{0}))} \leq \left(\frac{r_{2}}{r_{1}} \right)^{n+1}e^{\frac{K+ \vartheta^{2}}{6}(r^{2}_{2}-r^{2}_{1})}, \label{volcomp}
\end{equation}
where $\vartheta$ is defined by (\ref{supS}).
\end{lem}

By (\ref{volcomp}), we can get the following volume comparison
\be
\frac{m(B_{r_{2}}(x_0))}{m(B_{r_{1}}(x_0))}\leq \left(\frac{r_{2}}{r_{1}}\right)^{n+1}e^{\frac{K+ \vartheta^{2}}{6}R^2}, \label{doubvol}
\ee
which implies the doubling volume property of $(M, F, m)$, that is, there is a uniform constant $D_{0}$ such that $m(B_{2r}(x_0))\leq D_{0} m(B_{r}(x_0))$ for any $x_0 \in M$ and $0< r < R/2$.

From  (\ref{doubvol}), we can get the following local $p$-Poincar\'{e} inequality. For the details, please see Corollary 4.6 in \cite{ChengF2}.

\begin{lem}{\rm(\cite{ChengF2})}\label{global-P-ineq}
Let $(M, F, m)$ be an $n$-dimensional forward complete Finsler measure space with finite reversibility $\Lambda$. Assume that ${\rm Ric}_{\infty}\geq -K$ for some $K\geq 0$. Fix $1\leq p<\infty$, then there exist positive constants $d_{i}=d_{i}(p, n, \Lambda)(i=1,2)$ depending only on $p$, $n$ and the reversibility $\Lambda$ of $F$, such that
\be
\int_{B_{R}}\left|u-\bar{u}\right|^p dm \leq d_1 e^{d_2 (K+ \vartheta^2)R^2} R^p \int_{B_{R}} F^{*p}(du) dm  \label{pi12}
\ee
for $u \in W_{\mathrm{loc}}^{1,p}(M)$, where $\bar{u}:=\frac{\int_{B_{R}}u dm} {m(B_{R})}$.
\end{lem}

Given a Finsler manifold $(M, F)$, we say that a family of open subsets ${\{U_{i}\}}_{i \in I}$ of $M$ is a uniformly locally finite covering of $M$ if ${\{U_{i}\}}_{i \in I}$ is a covering of $M$ and there exists an integer $N$ such that each point $x \in M$ has a neighborhood which intersects at most $N$ open subsets of ${\{U_{i}\}}_{i \in I}$, where $I$ is an index set. From Lemma \ref{volume} and following the arguments of Lemma 1.1 in \cite{Heby} and Lemma 5.1 in \cite{Xia},  we can obtain the following lemma.

\begin{lem} \label{unicov}
Let $(M, F, m)$ be an $n$-dimensional forward complete Finsler manifold with finite reversibility $\Lambda$. Assume that ${\rm Ric}_{\infty} \geq -K$ for some $K \geq 0$. Let $ R > 0$ be given. Then, there exist a sequence $(x_{i})_{i\in I}$ of points of $M$, such that for any $0 < r  \leq R$,
\begin{itemize}
  \item[{\rm (1)}]
  the family $\{B_{r}(x_{i})\}_{i \in I}$ is a uniformly locally finite covering of $M$, and there is an upper bound for $N$ depending on $n, r, R, K, \vartheta$ and $\Lambda$;
  \item[{\rm (2)}]
  for any $i < j$, $B_{r/2\Lambda}(x_{i})\cap B_{r/2\Lambda}(x_{j}) = \emptyset$.
\end{itemize}
\end{lem}

From Lemma \ref{volume}, Lemma \ref{global-P-ineq} and Lemma \ref{unicov}, we can obtain a global $p$-Poincar\'{e} inequality stated in the following theorem.

\begin{thm} \label{Poincare-ineq}
Let $(M, F, m)$ be an $n$-dimensional forward complete Finsler manifold with finite reversibility $\Lambda$. Assume that ${\rm Ric}_{\infty} \geq -K$ for some $K \geq 0$. Let $R > 0$ be some positive real number and $p\geq 1$. Then, there exists a positive constant $C=C(n, K, R, \vartheta, \Lambda)$, such that for any $r\in (0, R)$ and $u \in W^{1,p}(M)$,
\begin{equation}
\int_{M} |u - \bar{u}_{r}|^{p} \,dm \leq C  r^{p} \int_{M} F^{*p}(du) \,dm  , \label{Poincare-1}
\end{equation}
where $\bar{u}_{r}=\bar{u}_{r}(x):= \frac{1}{m(B_{r}(x))} \int_{B_{r}(x)} u \,dm $, $x\in M$.
\end{thm}
\vskip 3mm

\begin{proof}
 By Lemma \ref{global-P-ineq}, for any $x \in M$, any $r \in (0, R)$ and any $u \in W^{1,p}_{\rm loc}(M)$,
\begin{equation}{\label{Poincare-2}}
\int_{B_{r}(x)} |u - \bar{u}_{r}|^{p} \,dm \leq d_{1}e^{d_{2}(K + \vartheta^{2})r^{2}} r^{p} \int_{B_{r}(x)} F^{*p}(du) \,dm.
\end{equation}
Let $r \in (0, R)$ be given, and let $(x_{i})_{i \in I}$ be a sequence of points of M such that
\begin{equation*}
M=\bigcup\limits_{i \in I} B_{r}(x_{i}) \ \ \text{and} \  B_{r/2\Lambda}(x_{i})\cap B_{r/2\Lambda}(x_{j})= \emptyset \ \ \text{if} \ \ i<j.
\end{equation*}
Then, by H\"{o}lder inequality, we have
\beq
\int_{M} |u - \bar{u}_{r}|^{p} \,dm & \leq & \sum\limits_{i \in I} \int_{B_{r}(x_{i})} |u - \bar{u}_{r}|^{p} \,dm   \nonumber \\
& \leq & \sum\limits_{i \in I} \int_{B_{r}(x_{i})} \left(|u - \bar{u}_{r}(x_{i})|  + |\bar{u}_{r}(x_{i}) - \bar{u}_{(\Lambda+1)r}(x_{i})| +|\bar{u}_{r} - \bar{u}_{(\Lambda+1)r}(x_{i})|\right)^{p} \,dm  \nonumber\\
& \leq & 3^{p-1}\left(\sum\limits_{i \in I} \int_{B_{r}(x_{i})} |u - \bar{u}_{r}(x_{i})|^{p}\,dm  + \sum\limits_{i \in I} \int_{B_{r}(x_{i})} |\bar{u}_{r}(x_{i}) - \bar{u}_{(\Lambda+1)r}(x_{i})|^{p}\,dm \right. \nonumber\\
& & \left.+ \sum\limits_{i \in I} \int_{B_{r}(x_{i})}  |\bar{u}_{r} - \bar{u}_{(\Lambda+1)r}(x_{i})|^{p} \,dm\right), \label{th43.4-1}
\eeq
where we have used the inequality $(a+b+c)^{p} \leq 3^{p-1}(a^{p}+b^{p}+c^{p})$ for $a,b,c\geq 0$ in the third inequality.

Next, by (\ref{Poincare-2}), we get that
\beq
\sum\limits_{i \in I}  \int_{B_{r}(x_{i})} |u - \bar{u}_{r}(x_{i})|^{p} \,dm & \leq & d_{1} e^{d_{2}(K+\vartheta^{2})r^{2}}  r^{p} \sum\limits_{i \in I} \int_{B_{r}(x_{i})} F^{*p}(du) \,dm  \nonumber\\
& \leq & Nd_{1} e^{d_{2}(K+ \vartheta^{2})r^{2}}  r^{p} \int_{M} F^{*p}(du) \,dm.  \label{th43.4-2}
\eeq
Furthermore,
\beq
\sum\limits_{i \in I} \int_{B_{r}(x_{i})} |\bar{u}_{r}(x_{i}) - \bar{u}_{(\Lambda+1)r}(x_{i})|^{p} \,dm & = & \sum\limits_{i \in I} m(B_{r}(x_{i})) |\bar{u}_{r}(x_{i}) - \bar{u}_{(\Lambda+1)r}(x_{i})|^{p} \nonumber \\
& \leq & \sum\limits_{i \in I}m(B_{r}(x_{i}))^{1-p} \left(\int_{B_{r}(x_{i})} |u - \bar{u}_{(\Lambda+1)r}(x_{i})| dm \right)^{p}.    \label{th43.4-3}
\eeq
By H\"{o}lder inequality  again, we have
$$
\left(\int_{B_{r}(x_{i})} |u - \bar{u}_{(\Lambda+1)r}(x_{i})| \,dm \right)^{p} \leq \int_{B_{r}(x_{i})} |u - \bar{u}_{(\Lambda+1)r}(x_{i})|^{p} \,dm\cdot m(B_{r}(x_{i}))^{p-1}.
$$
Plugging the above inequality into (\ref{th43.4-3}) and by (\ref{Poincare-2}), we obtain
\beq
\sum\limits_{i \in I} \int_{B_{r}(x_{i})} |\bar{u}_{r}(x_{i}) - \bar{u}_{(\Lambda+1)r}(x_{i})|^{p} \,dm & \leq & \sum\limits_{i \in I} \int_{B_{(\Lambda+1)r}(x_{i})} |u- \bar{u}_{(\Lambda+1)r}(x_{i})|^{p} \,dm      \nonumber\\
& \leq & N (\Lambda+1)^{p}d_{1} e^{d_{2}(K+\vartheta^{2})(\Lambda +1)^{2}r^{2}}  r^{p}  \int_{M} F^{*p}(du) \,dm.   \label{th43.4-4}
\eeq
Independently, we have
\beq
&&\sum\limits_{i \in I} \int_{B_{r}(x_{i})} |\bar{u}_{r} - \bar{u}_{(\Lambda+1)r}(x_{i})|^{p} \,dm \nonumber\\
&& \leq  \sum\limits_{i \in I} \int_{x \in B_{r}(x_{i})} \left\{\frac{1}{m(B_{r}(x))} \int_{z \in B_{r}(x)} |u(z) - \bar{u}_{(\Lambda+1)r}(x_{i})| \,dm(z)\right\}^{p}  \,dm(x) \nonumber \\
&& = \sum\limits_{i \in I} \int_{x \in B_{r}(x_{i})} \left\{\left(\frac{1}{m(B_{r}(x))}\right)^{p} \left(\int_{z \in B_{r}(x)} |u(z) - \bar{u}_{(\Lambda+1)r}(x_{i})| \,dm(z)\right)^{p}\right\}  \,dm(x).\label{th43.4-5}
\eeq
 Again, by using H\"{o}lder inequality,
 $$
 \left(\int_{z \in B_{r}(x)} |u(z) - \bar{u}_{(\Lambda+1)r}(x_{i})| \,dm(z)\right)^{p} \leq \int_{z \in B_{r}(x)} |u(z) - \bar{u}_{(\Lambda+1)r}(x_{i})|^{p} \,dm(z) \cdot m(B_{r}(x))^{p-1}.
 $$
Plugging the above ineuqality into (\ref{th43.4-5}) yields
\beqn
&& \sum\limits_{i \in I} \int_{B_{r}(x_{i})} |\bar{u}_{r} - \bar{u}_{(\Lambda+1)r}(x_{i})|^{p} \,dm\\
&& \leq  \sum\limits_{i \in I} \int_{x \in B_{r}(x_{i})} \left\{\frac{1}{m(B_{r}(x))} \int_{z \in B_{r}(x)} |u(z) - \bar{u}_{(\Lambda+1)r}(x_{i})|^{p} \,dm(z)\right\}  \,dm(x) \\
&& \leq  \sum\limits_{i \in I} \int_{x \in B_{r}(x_{i})} \left\{\frac{1}{m(B_{r}(x))} \int_{z \in B_{(\Lambda+1)r}(x_{i})} |u(z) - \bar{u}_{(\Lambda+1)r}(x_{i})|^{p} \,dm(z)\right\}  \,dm(x) \\
&& \leq  \sum\limits_{i \in I} \int_{z \in B_{(\Lambda+1)r}(x_{i})} |u(z) - \bar{u}_{(\Lambda+1)r}(x_{i})|^{p} \, dm(z) \cdot \int_{x \in B_{r}(x_{i})} \frac{1}{m(B_{r}(x))} \, dm(x),
\eeqn
where we have used the fact that $B_{r}(x)\subset B_{(\Lambda +1)r}(x_{i})$ for $x \in B_{r}(x_{i})$ in the second inequality.  Using (\ref{Poincare-2}), we obtian
$$
\int_{z \in B_{(\Lambda+1)r}(x_{i})} |u(z) - \bar{u}_{(\Lambda+1)r}(x_{i})|^{p} \, dm(z) \leq (\Lambda+1)^{p}d_{1} e^{d_{2}(K+ \vartheta^{2})(\Lambda+1)^{2}r^{2}}  r^{p} \int_{B_{(\Lambda+1)r}(x_{i})} F^{*p}(du) \,dm.
$$
On the other hand, by volume comparison (\ref{doubvol}), we have the following for $0< r < (\Lambda +1)r < 2\Lambda R$
$$
\frac{1}{m(B_{r}(x))} \leq \frac{(\Lambda+1)^{n+1}e^{\frac{(K+\vartheta^{2})(2 \Lambda)^{2}R^{2}}{6}}}{m(B_{(\Lambda+1)r}(x))}.
$$
Note that $x \in B_{r}(x_{i})$ implies that $B_{r}(x_{i}) \subset B_{(\Lambda+1)r}(x)$. Hence
\beqn
\int_{x \in B_{r}(x_{i})} \frac{1}{m(B_{r}(x))} \, dm(x) & \leq & \int_{B_{r}(x_{i})}\frac{(\Lambda+1)^{n+1}e^{\frac{(K+ \vartheta^{2})(2\Lambda)^{2}R^{2}}{6}}}{m(B_{(\Lambda+1)r}(x))}  \,dm(x) \\
& \leq & \int_{B_{r}(x_{i})}\frac{(\Lambda+1)^{n+1}e^{\frac{(K+ \vartheta^{2})(2\Lambda)^{2}R^{2}}{6}}}{m(B_{r}(x_{i}))}  \,dm(x) \\
& = & (\Lambda+1)^{n+1}e^{\frac{(K+ \vartheta^{2})(2\Lambda)^{2}R^{2}}{6}}.
\eeqn
Thus, we have
\beq
&& \sum\limits_{i \in I} \int_{B_{r}(x_{i})} |\bar{u}_{r} - \bar{u}_{(\Lambda+1)r}(x_{i})|^{p} \,dm \nonumber \\
&& \leq N (\Lambda+1)^{p+n+1}d_{1}e^{d_{2}(K+ \vartheta^{2})(\Lambda+1)^{2}R^{2}} \ e^{\frac{(K+ \vartheta^{2})(2\Lambda)^{2}R^{2}}{6}}  r^{p} \int_{M} F^{*p}(du) \, dm. \label{th43.4-6}
\eeq
Finally, substituting (\ref{th43.4-2}), (\ref{th43.4-4}), (\ref{th43.4-6}) into (\ref{th43.4-1}) yields
$$
\int_{M} | u - \bar{u}_{r}|^{p} \, dm \leq C  r^{p} \int_{M} F^{*p}(du) \,dm,
$$
where $C=C(n,  K, R, \vartheta , \Lambda)>0$ is a constant.  This completes the proof of Theorem \ref{Poincare-ineq}.
\end{proof}

\section{A $(p, q)$-Sobolev inequality}\label{Ssob}

This section is devoted to establish  a global optimal $(p, q)$-Sobolev inequality and prove Theorem \ref{sob5}. For this aim, we need the following lemma which is Lemma 5.3 in \cite{Xia}.

\begin{lem}{\rm (\cite{Xia})}\label{sobll} Let $(M, F, m)$ be an $n$-dimensional forward complete Finsler manifold with finite reversibility $\Lambda$. Assume that, for any $u \in W^{1,1}(M)$, $\|u\|_{L^{\frac{\nu}{\nu-1}}} \leq \tilde{c}\|u\|_{1,1}$ for some $\tilde{c} > 0$ and $\nu>1$. Then, for any $u \in W^{1,q}(M) (1 \leq q < \nu)$, there exist positive constants $c_{i}=c_{i}(\tilde{c}, q, \nu, \Lambda)( i = 1, 2)$ such that
\begin{equation}{\label{sob2}}
\left(\int_{M} | u |^{p} \, dm\right)^{\frac{1}{p}} \leq c_{1} \left(\int_{M} | u |^{q} \, dm\right)^{\frac{1}{q}} + c_{2} \left(\int_{M} [F^{*}(du)]^{q} \, dm\right)^{\frac{1}{q}}
\end{equation}
holds, where $\frac{1}{p} = \frac{1}{q}-\frac{1}{\nu}$.
\end{lem}
Based on Theorem \ref{Poincare-ineq} and Lemma \ref{sobll}, we have the following theorem.

\begin{thm}{\label{sob3}}
Let $(M, F, m)$ be an $n$-dimensional forward complete Finsler manifold with finite reversibility $\Lambda$. Assume that ${\rm Ric}_{\infty} \geq -K$ for some $K \geq 0$, $d= {\rm Diam}(M)< \infty$ and $m_{0}=m(M)>0$. Then, for any $\nu=\nu(n)>n+1$ and $u \in W^{1,q}(M) \ (1\leq q < \nu)$, there exist positive constants $c_{i}=c_{i}(\Lambda, n, K, m_{0}, d, \vartheta)( i = 3, 4)$ such that
\begin{equation}{\label{sob4}}
\left(\int_{M} | u |^{p} \, dm\right)^{\frac{1}{p}} \leq c_{3} \left(\int_{M} | u |^{q} \, dm\right)^{\frac{1}{q}} + c_{4} \left(\int_{M} [F^{*}(du)]^{q} \, dm\right)^{\frac{1}{q}}
\end{equation}
where $\frac{1}{p} = \frac{1}{q}-\frac{1}{\nu}$.
\end{thm}

\begin{proof}
By Theorem \ref{Poincare-ineq}, Lemma \ref{sobll} and  doubling volume property (\ref{doubvol}), similar to the arguments of Theorem 5.1 in \cite{Xia}, we can completes the proof the theorem.
\end{proof}

\vskip 2mm

Now we are in the position to prove Theorem \ref{sob5}.
\vskip 2mm

\noindent{\it\bf Proof of Theorem \ref{sob5}:} \ We will give the proof according to the case when $p\geq 2$ or $1< p < 2$.

{\bf Case 1:} \ $p \geq 2$. \ As a starting point, we first prove that, for any $u \in L^{p}(M)$
\begin{equation}
\left(\int_{M} | u |^{p} \, dm\right)^{\frac{2}{p}} \leq m_{0}^{-\frac{2(p-1)}{p}} \left(\int_{M} u  \, dm\right)^{2} + (p-1) \left( \int_{M} |u- \bar{u}|^{p}
\,dm \right)^{\frac{2}{p}}, \label{sob7}
\end{equation}
where $\bar{u} = \frac{1}{m_{0}} \int_{M} u \, dm$.
As one can easily check, it is enough to prove (\ref{sob7}) for any $u \in C^{0}(M)$. Further, for homogeneity reason, and since the inequality is obviously satisfied if $\int_{M} u \, dm= 0$, we can restrict ourselves to the functions $u \in C^{\infty}(M)$ satisfying $\int_{M}u \ dm = m_{0}$.  Now, for such functions, one can write that
$$
u=1+tv,
$$
where $t\geq 0$ and $v \in C^{0}(M)$  such that $\int_{M} v \,dm = 0$ and $\int_{M} v^{2} \, dm =1$. Then (\ref{sob7}) is equivalent to
\begin{equation}
\left(\int_{M} | 1+ tv |^{p} \, dm\right)^{\frac{2}{p}} \leq m_{0}^{\frac{2}{p}} + t^{2}(p-1) \left( \int_{M} |v|^{p}
\,dm \right)^{\frac{2}{p}}.   \label{sob8}
\end{equation}

Let
$$
\varphi(t)=\left( \int_{M} | 1+ tv|^{p} \, dm\right)^{\frac{2}{p}}.
$$
Then $\varphi(0)= m_{0}^{\frac{2}{p}}$, $\varphi'(0)=0$. Further, a simple computation shows that
\beqn
\varphi''(t) & = & 2p\left(\frac{2}{p}-1\right) \left(\int_{M} | 1+ tv|^{p-1}v \ {\rm sgn}(1+tv) \, dm \right)^{2}\left(\int_{M} | 1+ tv|^{p} \, dm \right)^{\frac{2}{p}-2}\\
& & + 2(p-1)\left(\int_{M} | 1+ tv|^{p} \, dm \right)^{\frac{2}{p}-1} \int_{M} | 1+ tv|^{p-2}v^{2} \, dm.
\eeqn
Because of $p \geq 2$, the first term on the right-hand side of above equality is non-positive. On the other hand, by H\"{o}lder inequality, we have
\[
\int_{M} |1+t v|^{p-2} v^{2} dm \leq \left(\int_{M} |1+t v|^{p} d m \right)^{\frac{p-2}{p}}\left(\int_{M} |v|^p d m\right)^{\frac{2}{p}}.
\]
Then we obtain the following
\[
\varphi''(t)  \leq   2(p-1)\left( \int_{M} |v|^{p} \, dm \right)^{\frac{2}{p}},
\]
from which, we have
\beqn
\varphi'(t)  &\leq &  2(p-1)\left( \int_{M} |v|^{p} \, dm \right)^{\frac{2}{p}}t,\\
\varphi(t)  &\leq & m_{0}^{\frac{2}{p}}+ (p-1)\left( \int_{M} |v|^{p} \, dm \right)^{\frac{2}{p}}t^{2}.
\eeqn
Thus we conclude that (\ref{sob7}) holds.

Next, by Theorem \ref{sob3} and Lemma \ref{global-P-ineq} with $R = d$, we have
\beq
\left(\int_{M} | u - \bar{u} |^{p} \, dm\right)^{\frac{1}{p}} &\leq & c_{3} \left(\int_{M} | u - \bar{u} |^{q} \, dm\right)^{\frac{1}{q}} + c_{4} \left(\int_{M} [F^{*}(du)]^{q} \, dm\right)^{\frac{1}{q}} \nonumber \\
 & \leq & c_{3}\left(d_{1}e^{d_{2}(K+\vartheta^{2})d^{2}}d^{q}\right)^{\frac{1}{q}} \left(\int_{M} [F^{*}(du)]^{q} \, dm\right)^{\frac{1}{q}}+ c_{4} \left(\int_{M} [F^{*}(du)]^{q} \, dm\right)^{\frac{1}{q}}\nonumber \\
&:=& B \left(\int_{M} [F^{*}(du)]^{q} \, dm\right)^{\frac{1}{q}},  \label{sob9}
\eeq
where $B=B(\Lambda, n, K, m_{0}, d, \vartheta)$.

For $q \in [1, 2]$, $\frac{q}{2} \leq 1$ and for any $a,b \geq 0$, $(a+b)^{\frac{q}{2}}\leq a^{\frac{q}{2}}+ b^{\frac{q}{2}}$.  From (\ref{sob7}), (\ref{sob9}) and by H\"{o}lder inequality, we obtain
\beqn
\left(\int_{M} | u |^{p} \, dm\right)^{\frac{q}{p}} & \leq & m_{0}^{-\frac{q(p-1)}{p}} \left(\int_{M}  |u | \, dm\right)^{q} + (p-1)^{\frac{q}{2}} \left( \int_{M} |u- \bar{u}|^{p}
\,dm \right)^{\frac{q}{p}}\\
& \leq & m_{0}^{-\frac{q(p-1)}{p}+q-1} \int_{M} |u|^{q} \, dm + (p-1)^{\frac{q}{2}}B^{q} \int_{M} [F^{*}(du)]^{q} \,dm ,
\eeqn
where we have used the fact that
\be
\left(\int_{M}  |u|  \, dm\right)^{q}\leq m_{0}^{q-1}\int_{M} |u|^{q} d m . \label{Holder}
\ee

Note that
$$
-\frac{q(p-1)}{p}+q-1 = -\frac{q}{\nu}.
$$
Hence, we have
$$
\left(\int_{M} | u |^{p} \, dm\right)^{\frac{q}{p}} \leq m_{0}^{-\frac{q}{\nu}} \int_{M} | u |^{q} \, dm + A \int_{M} [F^{*}(du)]^{q} \,dm ,
$$
where $A:= (p-1)^{\frac{q}{2}}B^{q}$. It is just (\ref{sob6}) in the case that $p \geq 2$.

{\bf Case 2:}  $1 < p <2$. \ Fristly, we start with the proof that,  for any $u \in L^{p}(M)$ and $1 \leq q \leq p$,
\begin{equation}
\left(\int_{M} | u |^{p} \, dm\right)^{\frac{q}{p}} \leq m_{0}^{\frac{q}{p}-q} \left|\int_{M}  u  \, dm \right|^{q} + \left( 1+ p(p-1)^{p-1} \right)^{\frac{q}{p}} \left( \int_{M} |u- \bar{u}|^{p} \,dm \right)^{\frac{q}{p}}. \label{sob10}
\end{equation}
Obviously,  we just need to assume that $u \in C^{0}(M)$ and that $\int_{M} u \, dm \neq 0$ for (\ref{sob10}).  Write
$$
u=\bar{u}(1+v),
$$
where $v$ satisfies $\int_{M} v \, dm=0$. It is easy to see that
\beqn
\int_{M} | 1+ v |^{p} \, dm  & = & \int_{ \left\{v \geq 0\right\}} | 1+ v |^{p} \, dm + \int_{\left\{ -1 \leq v <0 \right\} } | 1+ v |^{p} \, dm\\
& & + \int_{\left\{ v<-1 \right\}} | 1+ v |^{p} \, dm .
\eeqn

By $1 < p < 2$, the following holds
\begin{equation*}
\begin{cases}
& (1+x)^{p} \leq 1+px+x^{p}, \ \   x \geq 0 , \\
& (1-x)^{p} \leq 1-px+x^{p}, \ \  0 \leq  x \leq 1 , \\
& (x-1)^{p} \leq x^{p}, \ \ \ \ \ \ \ \ \ \ \ \ \ \   x \geq 1.
\end{cases}
\end{equation*}
Hence
\beqn
\int_{M} | 1+ v |^{p} \, dm & \leq & \int_{ \left\{v \geq 0\right\}} \ \, dm + p\int_{\left\{ v \geq 0 \right\} } v \, dm + \int_{\left\{ v \geq 0 \right\}} v^{p} \, dm + \int_{ \left\{-1 \leq v < 0\right\}} \ \, dm\\
& & + p\int_{\left\{ -1 \leq v < 0 \right\} } v \, dm + \int_{\left\{ -1 \leq v < 0 \right\} } |v|^{p} \, dm+ \int_{\left\{ v < -1 \right\} } |v|^{p} \, dm.
\eeqn
Then we have
\beqn
\int_{M} | 1+ v |^{p} \, dm & \leq & m\left(\{v \geq -1\}\right) + \int_{M} |v|^{p} \, dm + p\int_{\left\{ v \geq -1 \right\}} v \, dm \\
& = & m\left(\{v \geq -1 \}\right) + \int_{M} |v|^{p} \, dm+ p\int_{M}  v \, dm - p\int_{\left\{ v < -1 \right\}}v \, dm\\
& = & m\left(\{v \geq -1\}\right) + \int_{M} |v|^{p} \, dm+ \int_{\left\{ v < -1 \right\} } |v| \, dm.
\eeqn
By H\"{o}lder inequality, we obtain
\begin{equation}{\label{sob11}}
\int_{M} |1+v|^{p} \, dm \leq  m\left(\{v \geq -1\}\right) + \int_{M} |v|^{p} \, dm + p \cdot m ( \left\{ v < -1 \right\} )^{\frac{p-1}{p}} \left(\int_{M} |v|^{p} \, dm \right)^{\frac{1}{p}}.
\end{equation}

Now, Let $\bar{m}_{0}:= m\left(\{v \geq -1\}\right)$ and for any $t \in [0,\ m_{0}]$, let
\[
f(t)= t+\| v\|^{p}_{L^{p}}+p\| v\|_{L^{p}}(m_{0}-t)^{\frac{p-1}{p}}.
\]
It is easy to check that $f(0)= \| v\|^{p}_{L^{p}}+p\| v\|_{L^{p}}m_{0}^{\frac{p-1}{p}}$ and
\beq
f' (t)&=& 1-(p-1)\| v\|_{L^{p}}(m_{0}-t)^{-\frac{1}{p}}, \\
f''(t)&=& -\frac{p-1}{p}\| v\|_{L^{p}}(m_{0}-t)^{-\frac{1}{p}-1}.\label{sob12}
\eeq
Therefore, we have $\lim\limits_{t \rightarrow m_{0}^{-}} f'(t)=-\infty$ and $f'$ is monotone decreasing  on $[0,\ m_{0}]$.

Assume that $f'(0) <0$.  Then
\begin{equation}{\label{sob13}}
m_{0}^{\frac{1}{p}} < (p-1)\| v\|_{L^{p}}
\end{equation}
and $f$ is monotone decreasing  on $[0,\ m_{0}]$. Thus we have $f(\bar{m}_{0})\leq f(0)$, that is
\[
 \bar{m}_{0}+ \| v\|^{p}_{L^{p}}+ p\| v\|_{L^{p}}(m_{0}- \bar{m}_{0})^{\frac{p-1}{p}} \leq   \| v\|^{p}_{L^{p}}+ p\| v\|_{L^{p}}m_{0}^{\frac{p-1}{p}},
\]
from which and by (\ref{sob11}), (\ref{sob13}), we get
\begin{equation}{\label{sob14}}
\int_{M} |1+v|^{p} \, dm \leq  \left(1+p(p-1)^{p-1} \right) \| v\|^{p}_{L^{p}}.
\end{equation}

Assume that $f'(0) \geq 0$. Then $m_{0}^{\frac{1}{p}} \geq (p-1)\| v\|_{L^{p}}$.  Let $t_{c}:= m_{0}-(p-1)^{p} \| v\|^{p}_{L^{p}}$. It is easy to check that $f'(t_{c})=0$ and $t_{c}$ is the maximal value point of $f(t)$ in $[0, {m}_{0}]$.
Hence we have $f(\bar{m}_{0}) \leq f(t_{c})$. By (\ref{sob11}),  this implies to
\begin{align}{\label{sob15}}
\int_{M} |1+v|^{p} \, dm  & \leq  m_{0} +  \| v\|^{p}_{L^{p}} + p(p-1)^{p-1} \| v\|^{p}_{L^{p}} \nonumber \\
& = m_{0} + \left( 1+p(p-1)^{p-1} \right) \| v\|^{p}_{L^{p}}.
\end{align}

In sum, by ({\ref{sob14}}) and ({\ref{sob15}}), we have
\begin{equation}
\int_{M} |1+v|^{p} \, dm \leq  m_{0} + \left(1+p(p-1)^{p-1} \right) \int_{M} | v|^{p} \, dm. \label{sob16}
\end{equation}

Since $\frac{q}{p} \leq 1$, we can obtain
$$
\left(\int_{M} |1+v|^{p} \, dm\right)^{\frac{q}{p}} \leq  m_{0}^{\frac{q}{p}} + \left(1+p(p-1)^{p-1} \right)^{\frac{q}{p}} \left(\int_{M} | v|^{p} \, dm \right)^{\frac{q}{p}}.
$$
Multiplying this inequality by $|\bar{u}|^{q}$, we get the following
\begin{equation}{\label{sob17}}
\left(\int_{M} |u|^{p} \, dm\right)^{\frac{q}{p}} \leq  m_{0}^{\frac{q}{p}-q}\left|\int_{M} u \, dm \right|^{q} + \left(1+p(p-1)^{p-1} \right)^{\frac{q}{p}} \left(\int_{M} | u- \bar{u} |^{p} \, dm \right)^{\frac{q}{p}}.
\end{equation}
It is just (\ref{sob10}).

Now, from (\ref{sob17}) and by (\ref{sob9}) and (\ref{Holder}), for any $u \in W^{1,q}(M)$, we have
$$
\left(\int_{M} |u|^{p} \, dm\right)^{\frac{q}{p}} \leq  m_{0}^{\frac{q}{p}-1}\int_{M} |u |^{q} \, dm  + \left(1+p(p-1)^{p-1} \right)^{\frac{q}{p}}B^{q} \int_{M} [F^{*}(du)] ^{q} \, dm.
$$
Because $\frac{q}{p}-1=-\frac{q}{\nu}$, we have
$$
\left(\int_{M} |u|^{p} \, dm\right)^{\frac{q}{p}} \leq  m_{0}^{-\frac{q}{\nu}}\int_{M} |u |^{q} \, dm  + A \int_{M} [F^{*}(du)] ^{q} \, dm
$$
where $A:= \left(1+p(p-1)^{p-1} \right)^{\frac{q}{p}}B^{q}$. It is just (\ref{sob6}) in the case that $1 < p <2$. This completes the proof of Theorem \ref{sob5}. \qed

\section{The Nash inequality}\label{nash}
In this section, we will prove the Nash inequality on forward complete Finsler manifold by the $p$-Poincar\'{e} inequality and volume comparison.
\begin{thm}{\label{nashs1}}
Let $(M, F, m)$ be an $n$-dimensional forward complete Finsler manifold with finite reversibility $\Lambda$. Assume that ${\rm Ric}_{\infty} \geq -K$ for some $K \geq 0$, $d= {\rm Diam}(M)< \infty$ and $m_{0}=m(M)>0$. Let $p \geq 1$. Then, for any $u\in W^{1,p}(M)$, there exist positive constants $D=D(n, K, \vartheta, \Lambda, d)$ such that
\begin{equation}
\|u\|_{L^{p}}^{2+\frac{2}{n+1}} \leq D(n, K, \vartheta, \Lambda, d)m_{0}^{\frac{2(1-p)}{(n+1)p}}\left(\|F(\nabla u)\|_{L^{p}}^{2}+\|u\|_{L^{p}}^{2}\right)\|u\|_{L^{1}}^{\frac{2}{n+1}}. \label{nash2}
\end{equation}
\end{thm}

\begin{proof}
It follows from Minkowiski inequality that, for any $u \in W^{1,p}(M)$,
\beq
\|u\|_{L^{p}} & \leq & \|u-\bar{u}_{r}\|_{L^{p}}+\|\bar{u}_{r}\|_{L^{p}}\nonumber\\
& := & {\rm I}+{\rm II}. \label{nash2}
\eeq
By  Theorem \ref{Poincare-ineq}, we have
\begin{equation}{\label{nash3}}
{\rm I} =\|u-\bar{u}_{r}\|_{L^{p}} \leq Cr\|F(\nabla u)\|_{L^{p}}.
\end{equation}
On the other hand,  for any $r \in (0, d)$, we have
\begin{equation}{\label{nash4}}
|\bar{u}_{r}(x)|  \leq \frac{1}{m(B_{r}(x))}\int_{M} |u| \, dm.
\end{equation}
By volume comparison (\ref{doubvol}), we can get
\be
\frac{m(B_{d}(x))}{m(B_{r}(x))}=\frac{m_{0}}{m(B_{r}(x))}\leq \left(\frac{d}{r}\right)^{n+1}e^{\frac{K+ \vartheta^{2}}{6}d^2}. \label{nashh}
\ee
Plugging (\ref{nashh}) into (\ref{nash4}) yields
$$
|\bar{u}_{r}(x)|  \leq \left(\frac{d}{r}\right)^{n+1}e^{\frac{K+ \vartheta^{2}}{6}d^2} m_{0}^{-1}\int_{M} |u| \, dm.
$$
Hence,
$$
\|\bar{u}_{r}\|_{\infty} \leq \left(\frac{d}{r}\right)^{n+1}e^{\frac{K+ \vartheta^{2}}{6}d^2} m_{0}^{-1} \|u\|_{L^{1}}.
$$
Consequently, it follows that
\beq
{\rm II} =\|\bar{u}_{r}\|_{L^{p}} & \leq &  \left(\|\bar{u}_{r}\|_{\infty}^{p} m_{0}\right)^{\frac{1}{p}}\nonumber\\
& \leq & \left(\frac{d}{r}\right)^{n+1}e^{\frac{K+ \vartheta^{2}}{6}d^2} m_{0}^{\frac{1}{p}-1} \|u\|_{L^{1}}. \label{nash5}
\eeq
Substituting (\ref{nash3}) and (\ref{nash5}) into (\ref{nash2}), we can get
\beq
\|u\|_{L^{p}} & \leq & \|u-\bar{u}_{r}\|_{L^{p}}+\|\bar{u}_{r}\|_{L^{p}}\nonumber\\
& \leq & Cr\|F(\nabla u)\|_{L^{p}}+ \left(\frac{d}{r}\right)^{n+1}e^{\frac{K+ \vartheta^{2}}{6}d^2} m_{0}^{\frac{1}{p}-1} \|u\|_{L^{1}}\nonumber\\
& \leq & Cr(\|F(\nabla u)\|_{L^{p}}+\|u\|_{L^{p}})+ \left(\frac{d}{r}\right)^{n+1}e^{\frac{K+ \vartheta^{2}}{6}d^2} m_{0}^{\frac{1}{p}-1} \|u\|_{L^{1}}.\label{nash6}
\eeq
Notice that $\|u\|_{L^{p}}$ in the left hand side of (\ref{nash6}) does not rely on $r$. Let
$$
g(r):=Cr(\|F(\nabla u)\|_{L^{p}}+\|u\|_{L^{p}})+ \left(\frac{d}{r}\right)^{n+1}e^{\frac{K+ \vartheta^{2}}{6}d^2} m_{0}^{\frac{1}{p}-1} \|u\|_{L^{1}}.
$$
It is easy to see that
\beqn
g'(r) &= & H_{1}-\theta H_{2}r^{-(\theta+1)}, \\
g''(r)&= &\theta(\theta+1)H_{2}r^{-(\theta+2)},
\eeqn
where $H_{1}= C(\|F(\nabla u)\|_{L^{p}}+\|u\|_{L^{p}})$, $H_{2}=d^{\theta}e^{\frac{K+ \vartheta^{2}}{6}d^2} m_{0}^{\frac{1}{p}-1} \|u\|_{L^{1}}$ and $\theta = n+1$. Therefore, we can see that $g'(r)$ is increasing in $(0, d)$ and $g'(r_{min})=0$, where $r_{min}=(\frac{H_{1}}{\theta H_{2}})^{-\frac{1}{\theta +1}}$. Plugging $r_{min}$ into the right hand side of (\ref{nash6}), we can conclude the following
\[
\|u\|_{L^{p}}^{2+\frac{2}{n+1}} \leq D(n, K, \vartheta, \Lambda, d)m_{0}^{\frac{2(1-p)}{(n+1)p}}(\|F(\nabla u)\|_{L^{p}}^{2}+\|u\|_{L^{p}}^{2})\|u\|_{L^{1}}^{\frac{2}{n+1}},
\]
where $D(n, K, \vartheta, \Lambda, d)>0$ is a constant. This completes the proof of Theorem \ref{nashs1} (that is, Theorem \ref{nashs}).
\end{proof}

\end{document}